\documentclass[11pt]{scrartcl}
\usepackage[english]{babel}
\usepackage[latin1]{inputenc}
\usepackage{amssymb,amsmath,exscale,graphicx,amsthm}
\newcommand{\abssec}[1]{\noindent\normalsize {\bfseries #1\quad }\ignorespaces}
\renewenvironment{abstract}{\abssec{Abstract}}{\par\vspace{.1in}}
\newenvironment{keywords}{\abssec{Key Words}}{\par\vspace{.1in}}
\newenvironment{AMS}{\abssec{AMS subject
  classification}}{\par\vspace{.1in}}

\theoremstyle{plain}
\newtheorem{theorem}{Theorem}
\newtheorem{corollary}[theorem]{Corollary}

\newtheorem{lemma}[theorem]{Lemma}

\theoremstyle{definition}
\newtheorem{remark}[theorem]{Remark}
\newtheorem{example}[theorem]{Example}

\numberwithin{equation}{section}

\DeclareMathOperator{\sign}{sign}
\DeclareMathOperator{\proj}{Proj}

\title{On the regularity of the solutions of Dirichlet optimal control problems in polygonal domains\thanks{The second author was partially
supported by the Spanish Ministerio de Ciencia e Innovaci\'on under projects MTM2011-22711.}}

\author{%
Thomas Apel\thanks{Universit\"at der Bundeswehr M\"unchen, 85577 Neubiberg, Germany, (Thomas.Apel@unibw.de)}
\and
Mariano Mateos\thanks{Departamento de Matem\'{a}ticas, E.P.I. de Gij\'on,
Universidad de Oviedo, Campus de Gij\'on, 33203 Gij\'on, Spain (mmateos@uniovi.es).}
\and
Johannes Pfefferer\thanks{Universit\"at der Bundeswehr M\"unchen, 85577 Neuibiberg, Germany, (Johannes.Pfefferer@unibw.de)}
\and
Arnd R\"osch\thanks{Universt\"at Duisburg-Essen, Fakultät für Mathematik,
Thea-Leymann-Straße 9,
D-45127 Essen, Germany (Arnd.Roesch@uni-due.de)}
}

\begin{document}

\maketitle

\begin{abstract}A linear quadratic Dirichlet control problem posed on a possibly non-convex polygonal domain is analyzed. Detailed regularity results are provided in classical Sobolev (Slobodetski\u\i) spaces. In particular, it is proved that in the presence of control constraints, the optimal control is continuous despite the non-convexity of the domain.
\end{abstract}



\pagestyle{myheadings}
\thispagestyle{plain}
\markboth{T. APEL, M. MATEOS, J. PFEFFERER AND A. R\"OSCH} {DIRICHLET CONTROL IN POLYGONAL DOMAINS}

\section{Introduction}
\label{S1}

The investigation of optimal control problems with partial
differential equations has been of increasing interest in the last
decades.
In this paper we will study the control problem
\[
\mbox{(P)}\left\{\begin{array}{l}\displaystyle \min
J(u)=\frac{1}{2}\int_{\Omega}(Su(x)-y_\Omega(x))^2dx
 + \frac{\nu}{2}\int_\Gamma u^2(x)\, d\sigma(x)\vspace{2mm}\\
    \mbox{subject to} \ \
    (Su,u)\in H^{1/2}(\Omega)\times
    L^2(\Gamma),\vspace{2mm}\\
    u\in U_{ad}=\{u\in L^2(\Gamma):\ a\leq u(x)\leq b \ \ \mbox{ for a.a. } \ x\in \Gamma\},
    \end{array}\right.
\]
where $Su$ is the solution $y$ of the 
state equation 
\begin{equation}
 -\Delta y   =  0  \mbox{ in }\Omega,\ 
 y = u  \mbox{ on } \Gamma,
\label{E1.1}
\end{equation}
the domain $\Omega\subset\mathbb{R}^2$ is bounded and polygonal,
$\Gamma$ is its boundary, $a<b$ and $\nu>0$ are real constants and $y_\Omega$ is a function whose precise regularity will be stated when necessary.
Note that, for $u\in
U_{ad}$, the state equation does not possess a variational solution in
general, so a very weak solution is considered (see Theorem \ref{T2.4}). We will discuss here
the regularities of the optimal state $\bar y$, the optimal control
$\bar u$ and the corresponding adjoint state $\bar\varphi$ which are
limited by singularities due to corners of the domain and due to the
presence of control constraints.

The classical Sobolev
(Slobodetski\u\i) spaces are denoted by $W^{t,p}(\Omega)$ and, in the
case of $p=2$, by $H^t(\Omega)$. As usual, for $t>0$, $W^{t,p}_0(\Omega)$ or $H^t_0(\Omega)$ will denote the closure respectively in $W^{t,p}(\Omega)$ or $H^t(\Omega)$ of $\mathcal{D}(\Omega)$, the space of infinitely differentiable functions with compact support in $\Omega$, and $W^{-t,q}(\Omega)$ with $q^{-1}+p^{-1}=1$ [resp. $H^{-t}(\Omega)$] is the dual space of $W^{t,p}_0(\Omega)$ [resp. $H^{t}_0(\Omega)$].

The seminal paper on Dirichlet control problems is the work by Casas and
Raymond \cite{Casas-Raymond2006}. They investigate the problem even
with a semilinear state equation. Assuming a convex polygonal domain
with maximal interior angle $\omega_1<\pi$, they prove $\bar u\in
W^{1-1/p,p}(\Gamma)$ and $\bar y\in W^{1,p}(\Omega)$ with
$p<p_\Omega=2/(2-\min\{\lambda_1,2\})$ and $\lambda_1=\pi/\omega_1$ both for the
case with and without control constraints. Note that $p_\Omega>2$ due to
the convexity of the domain.
May, Rannacher and Vexler \cite{MayRannacherVexler2013} consider the
unconstrained Dirichlet control problem with linear state equation
\eqref{E1.1} also in convex domains and derive $\bar u\in
H^{1-1/p}(\Gamma)$ and $\bar y\in H^{3/2-1/p}(\Omega)$ for $p<p_\Omega$.
Deckelnick, G\"unter and Hinze \cite{Decklenick-Gunther-Hinze2009}
focus on approximation issues in the case of smooth domains (class
$C^3$) in two and three space dimensions. The regularity is determined
by the box constraints since corner (or edge) singularities do not
occur. Finally we would like to mention the paper
\cite{OfPhanSteinbach2010} by Of, Phan, and Steinbach where the
control is searched in $H^{1/2}(\Omega)$ such that the state $\bar y\in
H^1(\Omega)$ satisfies the weak formulation; in this case the
regularity issues are less severe.

Due to the assumptions on the domain, the publications
\cite{Casas-Raymond2006}, \cite{Decklenick-Gunther-Hinze2009}, and
\cite{MayRannacherVexler2013} have in common that the adjoint problem
can be solved in $H^2(\Omega)\cap H^1_0(\Omega)$. For that reason
the very weak formulation of the state equation  is well
defined in these publications with test functions from $H^2(\Omega)\cap
H^1_0(\Omega)$.  This is not the case when non-convex domains are
considered. Instead, the very weak fomulation of the state equation
should be defined with test functions from $H^1_\Delta(\Omega) \cap
H^1_0(\Omega)$, where $H^1_\Delta(\Omega):=\{v\in H^1(\Omega): \Delta
v\in L^2(\Omega)\}$, see the paper by Casas, Mateos and Raymond
\cite[(A.16)]{Casas-Mateos-Raymond-2009} and also
\cite{Apel-Nicaise-Pfefferer2014} for further approaches how to
understand the solution of the Poisson equation with non-smooth
boundary data.

In the paper at hand we prove basic regularity results for the
solution of the state and adjoint state equations in Section \ref{sec:2}. We extend in Theorem \ref{T2.4} to non-convex domains the well-known $H^{1/2}(\Omega)$ regularity of the solution of the state equation \eqref{E1.1} and we prove in Theorem \ref{T2.7} that the maximum principle also holds for very weak solutions.
Our main regularity results for the variables in the optimal control problem are proved
in Section \ref{sec:3}. Among many other results, we prove that, in the presence of control constraints and under minimal regularity assumptions on the data ($y_\Omega\in L^{s^*}(\Omega)$, $s^*>2$), the optimal control is continuous despite the possible non-convexity of the domain (cf. Theorem \ref{T3.4}). The main idea is very simple: we compute explicitly the normal derivative of the singular part of the adjoint state and exploit the projection relation established by the first order optimality conditions. We also investigate the case $s^*=2$. This case has not been treated by the cited references for convex domains.
In Section \ref{sec:4} we prove that, for regular data, the regularity of the optimal solution is indeed slightly better. We give conditions for the control to be in $H^{3/2-\varepsilon}(\Omega)$ for all $\varepsilon>0$ (Corollary \ref{C4.2}), which is the best regularity we can expect under pointwise control constraints.
These results will be helpful to derive error estimates for  finite element approximations of problem (P). The numerical analysis will be carried out in a forthcoming paper.
The short Section \ref{sec:5} is
devoted to the unconstrained case. Whereas in many other control
problems the regularity of the unconstrained solution is better than
that of the constrained solution, we encounter here the phenomenon
that the constraint inhibits poles of the unconstrained solution. For
that reason the regularity of the constrained control is determined, up to some exceptional cases (cf. Remark \ref{R3.5}), by
the largest convex angle but the regularity of the unconstrained
control is determined by the overall largest angle (cf. Corollary \ref{C5.2}).


Finally we would like to remark that sometimes the state equation is
considered in the form
\begin{equation}\label{E1.3}
  -\Delta y   =  f  \mbox{ in }\Omega,\ y  =  u + g   \mbox{ on } \Gamma,
\end{equation}
but we can take the unique function $y_0$ which solves $-\Delta y_0=f$
in $\Omega$, $y_0=g$ on $\Gamma$, and replace $y:=y+y_0$,
$y_\Omega:=y_\Omega-y_0$ and recover problem (P) with \eqref{E1.1} for data sufficiently smooth.

\section{\label{sec:2}Notation and basic results for elliptic equations}
Let us denote by $M$ the number of sides of $\Gamma$ and $\{x_j\}_{j=1}^M$ its vertexes, ordered counterclockwise. For convenience denote also $x_0=x_M$ and $x_{M+1}=x_1$. We will denote by $\Gamma_j$ the side of $\Gamma$ connecting $x_{j}$ and $x_{j+1}$, and by $\omega_j\in (0,2\pi)$ the angle interior to $\Omega$ at $x_j$, i.e., the angle defined by $\Gamma_{j}$ and $\Gamma_{j-1}$, measured counterclockwise. Notice that  $\Gamma_{0}=\Gamma_M$. We will use $(r_j,\theta_j)$ as local polar coordinates at $x_j$, with $r_j=|x-x_j|$ and $\theta_j$ the angle defined by $\Gamma_j$ and the segment $[x_j,x]$.
In order to describe the regularity of the functions near the corners, we will introduce for every $j=1,\ldots,M$ the infinite cone
\[K_j=\{x\in\mathbb{R}^2:0<r_j,\,0<\theta_j<\omega_j\}\]
and a positive number $R_j$ such that the sets
\[N_j=\{x\in\mathbb{R}^2:0<r_j<2R_j,\,0<\theta_j<\omega_j\},\]
satisfy $N_j\subset\Omega$ for all $j$ and $N_i\cap N_j=\emptyset$ if $i\neq j$. 

For every $j=1,\ldots,M$ we will also consider \[\xi_j:\mathbb{R}^2\to[0,1]\] an infinitely differentiable cut-off function which is equal to $1$ in the set $\{x\in\mathbb{R}^2:r_j<R_j\}$ and equal to $0$ in the set $\{x\in\mathbb{R}^2:r_j>2R_j\}$.

We will denote by $z_f$ the solution of the homogeneous Dirichlet problem with distributed data
\begin{equation}\label{E2.1}-\Delta z=f\mbox{ in }\Omega,\ z=0\mbox{ on }\Gamma.\end{equation}
The regularity of $f$ and $z_f$, as well as in what sense the equation must be understood, will be stated when necessary.

For every $j=1,\ldots,M$ we will call $\lambda_j$ the leading singular exponent associated with the operator corresponding to the corner $x_j$. For the Laplace operator it is well known that $\lambda_j=\pi/\omega_j$.
For convenience we will suppose that $\lambda_1=\min\{\lambda_j:\ j=1,\ldots,M\}$.
In general, the maximum regularity in non-weighted Sobolev spaces of the solution of the Poisson problem for regular data will be given by the Sobolev exponents \[p_\Omega=\frac{2}{2-\min\{\lambda_1,2\}}\mbox{ and }t_\Omega=1+\lambda_1.\] This is, for $f\in C^\infty(\Omega)$, the solution $z_f$ of equation \eqref{E2.1}
will satisfy that $z\in W^{2,p}(\Omega)\cap H^{t}(\Omega)$, for all $p<p_\Omega$ (cf. Grisvard \cite[Th. 4.4.3.7]{Grisvard85}) and all $t<t_\Omega$ (cf. Dauge \cite[\S23.C]{Dauge1988}; see also  Grisvard \cite[Th. 2.4.3 and \S2.7]{Grisvard1992}). For less regular data $f\in W^{-1,q}(\Omega)$, the maximum regularity of the solution of the homogeneous Dirichlet problem is given by the exponent
\[p_D =\frac{2}{1-\min\{1,\lambda_1\}}.\]
This means that $z_f\in W^{1,q}(\Omega)$ if $q<p_D$ (cf. Dauge \cite[Theorem 1.1(i)]{Dauge1992}; see also Jerison and Kenig \cite[Thms. 0.5, 1.1, 1.3]{Jerison-Kenig1995}).
Further results  on the local regularity in each corner can be stated using weighted Sobolev spaces.

We will state now two lemmas collecting several regularity results that will be used later to state the regularity of the solution of the control problem.
The following lemma is a consequence of well known regularity results collected in the book by Grisvard \cite{Grisvard85}. It tells us accurately how the singularities arising from the corners behave for problems with regular data. We introduce the following sets in order to describe the singular behavior of the solution of the Poisson problem at the corners.

For $1< p <+\infty$ such that
\begin{equation}\label{E2.2}\frac{2(p-1)}{p\lambda_j}\not\in\mathbb{Z}\ \forall j\in\{1,\ldots,M\}
\end{equation}
 and $m\in \mathbb{Z}$, define
\begin{equation}\label{E2.3}
\mathbb{J}^m_p=\left\{j\in\{1,\ldots,M\}\mbox{ such that }
0<m\lambda_j < 2 -\frac{2}{p}\right\}.\end{equation}
The condition in \eqref{E2.2} is necessary to deduce the $W^{2,p}(\Omega)$-regularity of the regular part of the solution in the lemma below. The meaning of the sets $\mathbb{J}^m_p$ is the following:
\[\xi_j r_j^{m\lambda_j}\not\in W^{2,p}(\Omega)\mbox{ for all }j\in\mathbb{J}^m_p.\]
Notice that, for all $p<+\infty$,
$\mathbb{J}_p^m=\emptyset$ for
 $m\geq 4$ since $\lambda_j>1/2$ for every possible angle and $2-2/p<2$ for all $p<+\infty$. We also remark that $\mathbb{J}^3_p\subset\mathbb{J}^2_p\subset\mathbb{J}^1_p$.
\begin{lemma}\label{L2.1}Consider $1< p < +\infty$ satisfying \eqref{E2.2} and $f\in L^p(\Omega)$. Then there exist unique real numbers $(c_{j,m})_{j\in \mathbb{J}^m_p}$ and a unique solution $z_f\in H^{1}(\Omega)$ of problem \eqref{E2.1}
such that
\[z_f = z_{reg}+\sum_{m=1}^3\sum_{j\in \mathbb{J}^m_p} c_{j,m}\xi_j r_j^{m\lambda_j}\sin(m\lambda_j \theta_j)\]
where $z_{reg}\in W^{2,p}(\Omega)$ and the $\xi_j$ are the cut-off functions introduced above.
\end{lemma}
\begin{proof}The result is a direct consequence of \cite[Theorem 4.4.3.7]{Grisvard85}. This result can be applied since $z_f\in H^1(\Omega)$ thanks to \cite[Lemma 4.4.3.1]{Grisvard85}.
\qquad\end{proof}

\begin{corollary}\label{C2.2}If $f\in L^2(\Omega)$, then $z_f\in H^t(\Omega)$ for all $t<1+\lambda_1$ and  $\partial_n z_f\in L^2(\Gamma)$.
\end{corollary}
\begin{proof}
To prove this fact, we apply Lemma \ref{L2.1}: the regular part is in $H^2(\Omega)$; on the other hand, since $\lambda_j\geq\lambda_1>1/2$ and $m\geq 1$, then $\xi_j r_j^{m\lambda_j}\in H^{t}(\Omega)$ for all $t<1+\lambda_1$. From this we obtain that $z_f\in H^t(\Omega)$. Since we can choose $t>3/2$ we have $\partial_n z_f\in L^2(\Gamma)$.
\qquad\end{proof}

The next result states the regularity of the solution of problems with boundary data in $W^{1-1/p^*,p^*}(\Gamma)$ with $p^*\geq 2$. Let us recall (cf. \cite[Theorem 1.5.2.3]{Grisvard85}) that  the trace of any function $z\in W^{1,p^*}(\Omega)$ is in $W^{1-1/p^*,p^*}(\Gamma)$, the trace mapping is onto, and, for $p^*>2$, this space can be characterized as
\[W^{1-1/p^*,p^*}(\Gamma) = \left\{g\in \prod_{i = 1}^M W^{1-1/p^*,p^*}(\Gamma_i):\ g\mbox{ is continuous at every corner }x_j	 \right\}.\]
For $p^*=2$ the continuity requirement in the corners can be weakened to an integral condition, see \cite[Theorem 1.5.2.3(c)]{Grisvard85}.

\begin{lemma}\label{L2.3}Let $g\in W^{1-1/p^*,p^*}(\Gamma)$ for some $p^*\geq 2$. Then there exists a unique solution $z\in W^{1,p}(\Omega)$, for all $p\leq p^*$, $p<p_D$,
 of the equation
\[-\Delta z = 0\mbox{ in }\Omega,\ z=g\mbox{ on }\Gamma.\]
\end{lemma}
\begin{proof}
Due to the trace theorem, there exists a function $G\in W^{1,p^*}(\Omega)$ such that its trace is $\gamma G = g$. Moreover, we have that $\Delta G \in W^{-1,p^*}(\Omega)$. If we define $\zeta= z-G$, then it satisfies the boundary value problem
\[-\Delta \zeta = -\Delta G\mbox{ in }\Omega,\ \zeta = 0 \mbox{ on }\Gamma.\]
Let $p\leq p^*$, $p<p_D$. Then $\zeta\in W^{1,p}_0(\Omega)$ (see above or cf. Dauge \cite[Theorem 1.1(i)]{Dauge1992}; see also Jerison and Kenig \cite[Thms. 0.5, 1.1, 1.3]{Jerison-Kenig1995}) and hence so does $z\in W^{1,p}(\Omega)$.
\qquad\end{proof}

Since the space of controls is $L^2(\Gamma)$, the state equation must be understood in the transposition sense. Following  \cite{Berggren2004,Casas-Mateos-Raymond-2009}, for $u\in L^2(\Gamma)$, we will say that $y\in L^2(\Omega)$  is a solution of \eqref{E1.1} if for every $f\in L^2(\Omega)$
\begin{equation}\label{E2.4}\int_\Omega y f dx  = -\int_\Gamma u \partial_n z_f  d\sigma(x),
\end{equation}
where $z_f$ is defined in Lemma \ref{L2.1}. The definition makes sense thanks to Corollary \ref{C2.2}.

Existence, uniqueness and regularity of the solution in $y\in H^{1/2}(\Omega)$ if $\Omega$ is convex domains is a classical result and can be proved via transposition and interpolation. Let us briefly recall how this result is obtained for a convex domain. Consider the solution operator $S$ of \eqref{E1.1} with $Su=y$. Due to the Lemma \ref{L2.3} it is
\begin{equation}\label{E2.5}S\in\mathcal{L}(H^{1/2}(\Gamma),H^1(\Omega)).\end{equation} Using the classical transposition method (cf. \cite{Lions-Magenes68}) we also have that
\begin{equation}\label{E2.6}S\in\mathcal{L}(H^{-1/2}(\Gamma),L^2(\Omega)).\end{equation} The final result is obtained by interpolation using that
\[L^2(\Gamma)=[H^{1/2}(\Gamma),H^{-1/2}(\Gamma)]_{1/2}\mbox{ and }
H^{1/2}(\Omega)=[H^1(\Omega),L^2(\Omega)]_{1/2},\]
see e.g. \cite[Chap. 1, Eq. (2.41)]{Lions-Magenes68}) for the first result and notice that interpolation results are valid for spaces posed on Lipschitz domains (cf.  \cite[Theorem 12.2.7]{Brenner-Scott94} or \cite{Jerison-Kenig1995}).
We cannot use this scheme straightforward because \eqref{E2.6} uses explicitly that for every $f\in L^2(\Omega)$, $z_f\in H^2(\Omega)$, and this is not true for non-convex polygonal domains.

For problems posed on non-convex polygonal domains, Berggren \cite[Theorem 4.2]{Berggren2004} proves existence, uniqueness and regularity of the solution $y\in H^{t}(\Omega)$ for every $0<t<\epsilon\leq 1/2$ ($\epsilon$ depends on the domain).
We can also achieve $y\in H^{1/2}(\Omega)$ in non-convex domains. The proof of our following result uses interpolation spaces; a different proof by using integral operators is given in \cite{Apel-Nicaise-Pfefferer2014}.
\begin{theorem}\label{T2.4}For every $u\in L^2(\Gamma)$ there exists a unique solution $y\in H^{1/2}(\Omega)$ of \eqref{E1.1} and
\[\|y\|_{H^{1/2}(\Omega)}\leq C \|u\|_{L^2(\Gamma)}.\]
\end{theorem}
\begin{proof}
Notice that for $0<\varepsilon<1/2$, we also have that, for $\theta = 1/(1+2\varepsilon)$,
\[L^2(\Gamma)=[H^{1/2}(\Gamma),H^{-\varepsilon}(\Gamma)]_{\theta}\mbox{ and }
H^{1/2}(\Omega)=[H^1(\Omega),H^{1/2-\varepsilon}(\Omega)]_{\theta}\]
and therefore the result will be true if we can prove that $S\in\mathcal{L}(H^{-\varepsilon}(\Gamma),H^{1/2-\varepsilon}(\Omega))$ for some $\varepsilon>0$.

\medskip

Fix $0<\varepsilon < \min\{\lambda_1-1/2,1/2\}$.
For any $u\in H^{-\varepsilon}(\Gamma)$, we will say that $y=Su$ if for every $f\in L^2(\Omega)$
\begin{equation}\int_\Omega y f dx = - \langle u,\partial_n z_f\rangle_{H^{-\varepsilon}(\Gamma),H^{\varepsilon}(\Gamma)}.\label{E2.7}\end{equation}
Notice that since $\varepsilon<\lambda_1-1/2$ we have that $z_f\in H^{3/2+\varepsilon}(\Omega)$ and that $\partial_n z_f\in \prod_{j=1}^M H^{\varepsilon}(\Gamma_j)=H^{\varepsilon}(\Gamma)$ because $\varepsilon<1/2$ (cf. \cite[Theorem 1.5.2.3(a)]{Grisvard85}). Moreover, due to the trace theorem and elliptic regularity, we have that
\begin{equation}\label{E2.8}
\|\partial_n z_f\|_{H^\varepsilon(\Gamma)}\leq C \|f\|_{H^{\varepsilon-1/2}(\Omega)}.
\end{equation}
Notice also that if $u\in H^{1/2}(\Gamma)$ then the unique variational solution $y\in H^1(\Omega)$ of
\[-\Delta y = 0\mbox{ in }\Omega,\ y=u\mbox{ on }\Gamma\]
is a solution of \eqref{E2.7} and if $u\in L^2(\Gamma)$, then \eqref{E2.7} is the same as \eqref{E2.4}.

Let us prove uniqueness of the solution of \eqref{E2.7} in $L^2(\Omega)$ first. If $u=0$ and $y=Su\in L^2(\Omega)$, then, taking $f=y$ as test function in \eqref{E2.7} we get $\int_\Omega y^2dx =0$, and therefore $y\equiv 0$. Since the problem is linear, the solution is unique.

We next prove existence of a solution $y\in H^{1/2-\varepsilon}(\Omega)$ of \eqref{E2.7}. We know (cf. \cite[Theorem 1.4.2.4]{Grisvard85}) that $H^{1/2-\varepsilon}(\Omega)= H^{1/2-\varepsilon}_0(\Omega)$ and hence
$\big(H^{1/2-\varepsilon}(\Omega)\big)' =H^{\varepsilon-1/2}(\Omega)$.
Denote
\[\mathcal{F} = \{f \in L^2(\Omega)\mbox{ such that }
\|f\|_{H^{\varepsilon-1/2}(\Omega)}=1\}.\]
For any $u\in H^{1/2}(\Gamma)$, $y=Su$, we use that $L^2(\Omega)$ is dense in $H^{\varepsilon-1/2}(\Omega)$ and \eqref{E2.8} to obtain
\begin{eqnarray}
\|y\|_{H^{1/2-\varepsilon}(\Omega)}& =& \sup_{f\in \mathcal{F}}\langle f ,y\rangle_{H^{\varepsilon-1/2}(\Omega), H^{1/2-\varepsilon}(\Omega)} =\sup_{f\in \mathcal{F}} \int_\Omega y f dx\nonumber \\
&=& \sup_{f\in \mathcal{F}} -\langle u,\partial_n z_f\rangle_{H^{-\varepsilon}(\Gamma),H^{\varepsilon}(\Gamma)}\leq \sup_{f\in \mathcal{F}}  \|\partial_n z_f\|_{H^{\varepsilon}(\Gamma) } \|u\|_{H^{-\varepsilon}(\Gamma)}\nonumber \\
&\leq &C \sup_{f\in \mathcal{F}}  \|f\|_{H^{\varepsilon-1/2}(\Omega)} \|u\|_{H^{-\varepsilon}(\Gamma)} = C \|u\|_{H^{-\varepsilon}(\Gamma)}.
\label{E2.9}
\end{eqnarray}
Take a sequence $u_k\in H^{1/2}(\Gamma)$, $u_k\to u\in H^{-\varepsilon}(\Gamma)$, and let $y_k=Su_k$. We have just proved that
\[\|y_k-y_m\|_{H^{1/2-\varepsilon}(\Omega)}\leq \|u_k-u_m\|_{H^{-\varepsilon}(\Gamma)}\]
and therefore $y_k$ converges in $H^{1/2-\varepsilon}(\Omega)$ to some $y\in H^{1/2-\varepsilon}(\Omega)$ that is (the unique) solution of the equation:
\begin{eqnarray*}
\int_\Omega y fdx& = &\int_\Omega \lim_k y_k f dx = \lim_k\int_\Omega y_k f dx \\
&= &\lim_k -\langle u_k,\partial_n z_f \rangle_{H^{-\varepsilon}(\Gamma),H^{\varepsilon}(\Gamma)} =  -\langle \lim_k u_k,\partial_n z_f \rangle_{H^{-\varepsilon}(\Gamma),H^{\varepsilon}(\Gamma)} \\
& = &- \langle u,\partial_n z_f \rangle_{H^{-\varepsilon}(\Gamma),H^{\varepsilon}(\Gamma)}.
\end{eqnarray*}
Finally, \eqref{E2.9} implies that $S\in\mathcal{L}(H^{-\varepsilon}(\Gamma),H^{1/2-\varepsilon}(\Omega))$ and the proof is complete.
\qquad%
\end{proof}

The next result is rather technical. It will be used to describe precisely the structure of the optimal state in the proof of Theorem \ref{T3.6} below. In that result we will be able to write the control as the sum of a regular part and a singular part. We show in Lemma 2.5 how to solve the state equation for singular boundary data. Besides the usual regular and singular parts that we described in Lemma \ref{L2.1}, a new singular part of the solution arises from the boundary data.

 Define the jump functions at the corners
\begin{equation}\label{E2.10}
\chi_j = \left\{\begin{array}{rl}
1&\mbox{ on }\{x\in\partial K_j:\ \theta_j=0\}\\
-1&\mbox{ on }\{x\in\partial K_j:\ \theta_j=\omega_j\}.
\end{array}\right.
\end{equation}
Notice that $\chi_j=1$ on $\Gamma_j$ and $\chi_j=-1$ on $\Gamma_{j-1}$.
\begin{lemma}\label{L2.5}Consider any pair of subsets $\mathbb{H}^1,\ \mathbb{H}^2\subset\{1,\ldots,M\}$ and
real numbers $-1/2<\eta_{j,1}$ and $a_{j,1}$ for all $j\in \mathbb{H}^1$ and $0<\eta_{j,2}$ and $a_{j,2}$ for all $j\in\mathbb{H}^2$ such that $\eta_{j,n}/\lambda_j\not\in\mathbb{Z}$ for any $j\in\mathbb{H}^n$, $n=1,2$. Define
\[u = \sum_ {j\in \mathbb{H}^1} a_{j,1} \xi_j r_j^{\eta_{j,1}} +\sum_ {j\in \mathbb{H}^2} \chi_j a_{j,2} \xi_j r_j^{\eta_{j,2}} \mbox{ on }\Gamma,\]
take $p$ such that
\[1< p<\inf\left\{\frac{2}{1-\eta_{j,n}}:\ j\in \mathbb{H}^n,\ n=1,2 \mbox{ and }\eta_{j,n}<1\right\},\]
where we consider $\inf\emptyset=+\infty$ and define $\mathbb{J}^m_p$ as in \eqref{E2.3}.
Then there exist unique real numbers $(c_{j,m})_{j\in \mathbb{J}^m_p}$, $m=1,2,3$ and a unique
solution $y\in H^{1/2}(\Omega)$ of equation \eqref{E1.1}
such that
\begin{equation}y = y_{reg}+\sum_{n=1}^2\sum_{j\in \mathbb{H}^n} a_{j,n}\xi_j r_j^{\eta_{j,n}}s_{j,n}(\theta_j)  +\sum_{m=1}^3\sum_{j\in \mathbb{J}^m_p} c_{j,m}\xi_jr_j^{m\lambda_j}\sin(m\lambda_j \theta_j),\label{E2.11}\end{equation}
where
\begin{equation}\label{E2.12}s_{j,n}(\theta_j)=\frac{(-1)^{n+1}-\cos(\eta_j\omega_j)}{\sin(\eta_j\omega_j)}\sin\left( \eta_j\theta_j\right) + \cos \left( \eta_j\theta_j\right)  \end{equation}
and $y_{reg}\in W^{2,p}(\Omega)$. If, further, $\eta_{j,1}>0$ for all $j\in \mathbb{H}^1$, then $y\in H^1(\Omega)$.
\end{lemma}

\begin{proof}Since $\eta_{j,1}>-1/2$, we have that $u\in L^2(\Gamma)$ and $y\in H^{1/2}(\Omega)$ thanks to Theorem \ref{T2.4}. If also $\eta_{j,1}>0$, then $u\in H^{1/2}(\Gamma)$ and hence Lemma \ref{L2.3} gives us that $y\in H^1(\Omega)$.

Notice next that $\eta_{j,n}>-1/2$ implies that $2/(1-\eta_{j,n})>4/3$ and $p$ is well defined.

A direct computation shows that for $n=1,2$, $y_{j,n} = r_j^{\eta_{j,n}}s_{j,n}(\theta) $ are respectively the solutions of the problems
\[-\Delta y_{j,1} = 0\mbox{ in }K_j,\ y_{j,1}=r_j^{\eta_{j,1}}\mbox{ on }\partial K_j,\]
\[-\Delta y_{j,2} = 0\mbox{ in }K_j,\ y_{j,2}=\chi_j r_j^{\eta_{j,2}}\mbox{ on }\partial K_j.\]
Since $\Delta y_{j,1}=0$, we have that
\[\Delta\sum_{n=1}^2\sum_{j\in \mathbb{J}^n}\xi_jy_{j,n}=\sum_{j=1}^n\sum_{j\in \mathbb{J}^n}(y_{j,n}\Delta \xi_j+2\nabla y_{j,n}\nabla\xi_j).\]
Since $|\nabla y_{j,n}|\leq C r_j^{\eta_{j,n}-1}$, the condition imposed on $p$ implies that
\[f = \Delta\sum_{n=1}^2\sum_{j\in \mathbb{J}^n}a_{j,n}\xi_jy_{j,n}\in L^p(\Omega),\] and we can write $y = y_f+\sum_{n=1}^2\sum_{j\in \mathbb{H}^n}a_{j,n}\xi_jy_{j,n}$, where
\[-\Delta y_f = f\mbox{ in }\Omega,\ y_f=0\mbox{ on }\Gamma.\]
Applying Lemma \ref{L2.1} we obtain  that $y_f=y_{reg}+\sum_{m=1}^3\sum_{j\in \mathbb{J}^m_p}c_{j,m} \xi_jr_j^{m\lambda_j}\sin(m\lambda_j \theta_j)$ and the proof is complete.
\qquad\end{proof}

Although the maximum principle is a well known result for weak solutions of equation \eqref{E1.1} (see the celebrated paper by Stampacchia \cite{Stampacchia1965}), we have not been able to find a reference of its validity for solutions defined in the transposition sense \eqref{E2.4}. For the sake of completeness, we include such a result. First, we prove the following technical lemma.
\begin{lemma}\label{L2.6}Consider $f\in L^2(\Omega)$, $f\geq 0$ and let $z_f\in H^1_0(\Omega)$ be the solution of equation \eqref{E2.1}. Then $\partial_n z_f\leq 0$ a.e. on $\Gamma$.
\end{lemma}
\begin{proof}
Take $u\in C^\infty(\Gamma)$, $u\geq 0$. Thanks to Lemma \ref{L2.3}, the solution of equation \eqref{E1.1} satisfies  $y\in H^1(\Omega)$ and the maximum principle for weak solutions, as proved by Stampacchia \cite{Stampacchia1965}, holds. Therefore $y\geq 0$. Integration by parts shows then that
\[0\leq \int_\Omega y fdx = -\int_\Gamma u\partial_n z_f d\sigma(x)\]
and the result follows by the usual density argument.
\qquad\end{proof}
\begin{theorem}\label{T2.7}Consider $u\in L^\infty(\Gamma)$. Then the solution $y$ of equation \eqref{E1.1} belongs to $L^\infty(\Omega)$ and
\[\|y\|_{L^\infty(\Omega)}\leq \|u\|_{L^\infty(\Gamma)}.\]
\end{theorem}
\begin{proof}Define $K=\|u\|_{L^\infty(\Gamma)}$. We will prove that $y\leq K$ a.e. on $\Omega$, the proof for $-y\leq K$ being analogous.

We already know that $y\in H^{1/2}(\Omega)\hookrightarrow L^2(\Omega)$ (cf. Theorem \ref{T2.4}). Define $y_K=y-K$ and $y_K^+=\max(y_K,0)\in L^2(\Omega)$. For every $f\in L^2(\Omega)$ we deduce from \eqref{E2.1} that
\[\int_\Omega y_K f dx = -\int_\Gamma (u-K)\partial_n z_f d\sigma(x)\]
We take $f = y_K^+\geq 0$. For this choice of $f$, we  know from Lemma \ref{L2.6} that $\partial_n z_f\leq 0$ a.e. on $\Gamma$ and therefore
\[0\leq \int_\Omega (y_K^{+})^2dx =\int_\Omega y_K y_K^+dx =
-\int_\Gamma (u-K)\partial_n z_f d\sigma(x)\leq 0.\]
So $y_K^+\equiv 0$ and $y\leq K$.
\qquad\end{proof}
\section{\label{sec:3}Main regularity results for the control problem}
\setcounter{equation}{0}
The following result is standard and the proof can be found in \cite{Casas-Raymond2006}. Though in that reference only convex domains are taken into account, this result is independent of the convexity of the domain, once we have proved Corollary \ref{C2.2} and Theorem \ref{T2.4}. Here and in the rest of the paper, $\proj_{[a,b]}(c)=\min\{b,\max\{a,c\}\}$ for any real numbers $a,b,c$.

\begin{lemma}\label{T3.1}Suppose $y_\Omega\in L^2(\Omega)$. Then problem (P) has a unique solution $\bar u\in L^2(\Gamma)$ with related state $\bar y\in H^{1/2}(\Omega)$ and adjoint state $\bar\varphi\in H^1_0(\Omega)$. The following optimality system is satisfied:
\begin{eqnarray}\bar u(x) &=& \proj_{[a,b]}\left(\frac{1}{\nu}\partial_n\bar\varphi(x)\right)\mbox{ for a.e. }x\in \Gamma,
\label{E3.1}\\
-\Delta \bar y &=& 0\mbox{ in }\Omega,\; \bar y=\bar u\mbox{ on }\Gamma,
\label{E3.2}\\
-\Delta\bar \varphi &=& \bar y-y_\Omega \mbox{ in }\Omega,\; \bar \varphi=0\mbox{ on }\Gamma.
\label{E3.3}\end{eqnarray}
\end{lemma}
As in the proof of Theorem \ref{T2.4}, the solution of the state equation \eqref{E3.2} must be understood in the transposition sense, whereas the adjoint state equation \eqref{E3.3} has a variational solution.

Next we state a regularity result for the  adjoint state in the framework of classical Sobolev--Slobodetski\u\i{} spaces.  In the rest of this section we will suppose that $y_\Omega\in L^{s^*}(\Omega)$ where, $2\leq s^*<+\infty$ satisfies
\begin{equation}\label{E3.4}\frac{2(s^*-1)}{\lambda_j s^*}\not\in\mathbb{Z}\ \forall j\in\{1,\ldots,M\}.\end{equation}

\begin{theorem}\label{T3.2}There exist a unique function $\bar\varphi_{reg}\in W^{2,s^*}(\Omega)$ and unique real numbers $(\hat c_{j,m})_{j\in \mathbb{J}_{s^*}^m}$, where $\mathbb{J}_{s^*}^m$ is defined in \eqref{E2.3},  such that
\begin{equation}\label{E3.5}\bar\varphi = \bar\varphi_{reg} +\sum_{m=1}^3\sum_{j\in\mathbb{J}^m_{s^*}} \hat c_{j,m}\xi_j r_j^{m\lambda_j}\sin(m\lambda_j\theta_j).\end{equation}
\end{theorem}
\begin{proof}
Since the problem is control constrained,  $\bar u\in L^\infty(\Gamma)$, and by the maximum principle proved in Theorem \ref{T2.7}, $\bar y\in L^\infty(\Omega)$. Therefore, $\bar y-y_\Omega\in L^{s^*}(\Omega)$, and we can use Lemma \ref{L2.1}: since the related adjoint state $\bar\varphi$ is the solution of \eqref{E3.3}
we have that there exist unique $\bar\varphi_{reg}\in W^{2,s^*}(\Omega)$ and  $(\hat c_{j,m})_{j\in \mathbb{J}_s^m}$ such that
 relation \eqref{E3.5} holds.
\qquad\end{proof}

For any $s\geq 2$ we define the set
\[\mathbb{H}^1_s=\{j\in\mathbb{J}^1_{s}:\ \lambda_j>1\}\]
and for $s\geq 2$ and $m=2,3$
\[\mathbb{H}^m_s=\{j\in\mathbb{J}^m_{s}:\  \hat c_{j,1}=0\},\]
where the coefficients $\hat c_{j,m}$ are the coefficients obtained in Theorem \ref{T3.2}.
Notice that the indexes in $\mathbb{H}^1_s$ correspond to convex corners. The indexes in $\mathbb{H}^m_s$ correspond to those non-convex corners where the main part of the singularity of the adjoint state vanishes, and hence the behavior of the solution at those non-convex corners can be somehow compared to the behavior of the solution at the convex corners. Notice also that
\[\cup_{m=1}^3\mathbb{J}^m_s = \cup_{m=1}^3\mathbb{H}^m_s\bigcup
\{j\in\mathbb{J}^1_s:\lambda_j<1\mbox{ and }\hat c_{j,1}\neq 0\}\]
Consider also $p\geq 2$ such that for $m=1,2,3$
\begin{equation}\label{E3.6}
p\leq s^*,\ p<\frac{2}{2-m\lambda_j}\mbox{ if } j\in \mathbb{H}^m_{s^*}.
\end{equation}
This condition on $p$ appears in a natural way in the proof of Theorem \ref{T3.4}, see \eqref{E3.10}. If $s^*>2$, then we can choose $p>2$. With this choice we have that
\begin{lemma}\label{L3.3}Let $p$ satisfy \eqref{E3.6} and for $m=1,2,3$ consider $j\in \mathbb{H}^m_{s^*}$. Then
\begin{equation}\label{E3.7}\xi_j r_j^{m\lambda_j-1}\in W^{1-1/p,p}(\Gamma).
\end{equation}
\end{lemma}
\begin{proof}Take $j\in \cup_{m=1}^3\mathbb{H}^m_{s^*}$ and consider $N_j$ the bounded cone of radius $2R_j$ defined in Section \ref{sec:2}. We first prove that $u_j=r_j^{m\lambda_j-1}\in W^{1,p}(N_j)$.

Since $j\in \cup_{m=1}^3\mathbb{H}^m_{s^*}$, then $m\lambda_j>1$, so $m\lambda_j-1>0$ and $r_j^{m\lambda_j-1}\in C(\bar N_j)\subset L^p(N_j)$. On the other hand $|\nabla u_j| = (m\lambda_j-1)r_j^{m\lambda_j-2}$, and making the usual change of variables to polar coordinates, we have
\begin{eqnarray*} \int\int_{N_j}|\nabla u_j|^p dx &=& (m\lambda_j-1)^p\omega_j\int_0^{2R_j} r_j r_j^{(m\lambda_j-2)p}dr,
\end{eqnarray*}
the last integral being convergent if and only if $(m\lambda_j-2)p+1>-1$. Taking into account that $j\in \mathbb{J}^m_{s^*}$ implies $m\lambda_j<2$, the previous condition is fulfilled if and only if $p<\frac{2}{2-m\lambda_j}$, which is the assumption.

Relation \eqref{E3.7} now follows from the smoothness of the cut-off function, the trace theorem and the continuity of $\xi_j r_j^{m\lambda_j-1}$.
\qquad\end{proof}

\begin{theorem}\label{T3.4}Let $p$ satisfy \eqref{E3.6}. Then, the optimal control $\bar u$ belongs to
$W^{1-1/p,p}(\Gamma)$,
the optimal state $\bar y$ belongs to $W^{1,q}(\Omega)$ for all $q\leq p$, $q<p_D$. In particular, if $s^*>2$, both are continuous functions.
\end{theorem}

\begin{proof}
We will exploit the projection relation \eqref{E3.1} and the expression for the adjoint state obtained in \eqref{E3.5}.

Notice first that $\bar\varphi_{reg}\in W^{2,s^*}(\Omega)$ and $\bar\varphi_{reg}=0$ on $\Gamma$, so $\partial_n\bar\varphi_{reg}\in W^{1-1/s^*,s^*}(\Gamma)$ (cf. \cite[Lemma A.2]{Casas-Mateos-Raymond-2009} for the case $s^*=2$ or \cite{Casas-Gunther-Mateos2011} for the case $s^*>2$). Moreover, if $s^*>2$, then
 $\partial_n\bar\varphi_{reg}(x_j)=0$ on every corner the normal derivative of the regular part is a continuous function on $\Gamma$.

We are going to compute now the normal derivative of the singular part. For any $m\in\{1,2,3\}$ and $j\in \mathbb{J}^m_{s^*}$, we have $\partial_n \xi_j r_j^{m\lambda_j}\sin(m\lambda_j\theta_j)\in  C^\infty(\Gamma\setminus\{x_j\})$, so we have that for every compact set $\mathcal{K}\subset\Gamma\setminus\{x_j:\ j\in \mathbb{J}^1_{s^*}\}$
\[\partial_n\bar\varphi \in W^{1-1/s^*,s^*}(\mathcal{K}).\]
Near the corners, for $r_j<R_j$
 we have on $\Gamma_j$ (where $\theta_j=0$) that
\begin{eqnarray}
\partial_n r_j^{m\lambda_j}\sin(m\lambda_j\theta_j)\xi_j &=&
-\frac{1}{r_j}\partial_\theta r_j^{m\lambda_j}\sin(m\lambda_j\theta_j)\nonumber \\
& =& - m\lambda_j r_j^{m\lambda_j-1}\cos(m\lambda_j 0)
 =-m\lambda_j r_j^{m\lambda_j-1} \label{E3.8}
\end{eqnarray}
and on $\Gamma_{j-1}$ (where $\theta_j=\omega_j$) that
\begin{eqnarray}
\partial_n r_j^{m\lambda_j}\sin(m\lambda_j\theta_j)\xi_j &=&
\frac{1}{r_j}\partial_\theta r_j^{m\lambda_j}\sin(m\lambda_j\theta_j)\nonumber \\
&=&m\lambda_j r_j^{m\lambda_j-1}\cos\left(m\frac{\pi}{\omega_j} \omega_j\right) =(-1)^m m\lambda_j r_j^{m\lambda_j-1}.\label{E3.9}
\end{eqnarray}
Next we will distinguish two cases.

Case 1: if $j\in \cup_{m=1}^3\mathbb{H}^m_{s^*}$ then $m\lambda_j>1$ and hence
the limit of both expressions is zero as $r_j\to 0$.
Noticing \eqref{E3.7}, we have that the choice of the exponent $p$ made in \eqref{E3.6} gives us
\begin{equation}\label{E3.10}\partial_n\left(\bar\varphi_{reg}+
\sum_{m=1}^3\sum_{j\in \mathbb{H}^m_{s^*}} \hat c_{j,m}\xi_j r_j^{m\lambda_j}\sin(m\lambda_j\theta_j)\right)\in W^{1-1/p,p}(\Gamma).
\end{equation}
So far, we can deduce that, for every compact set $\mathcal{K}\subset \Gamma\setminus \{x_j:\ \lambda_j<1,
\hat c_{j,1}\neq 0\}$
\begin{equation}\label{E3.11}\partial_n\bar\varphi \in W^{1-1/p,p}(\mathcal{K}).\end{equation}

Case 2: Now $j\in \mathbb{J}^1_{s^*}$, $\lambda_j<1$, and $\hat c_{j,1}\neq 0$.
We have $(-1)^m=-1$ and $m\lambda_j-1<0$, therefore using expressions \eqref{E3.8} and \eqref{E3.9} we have
\[\lim_{x\to x_j}\partial_n \hat c_{j,1} r_j^{\lambda_j}\sin(\lambda_j\theta_j)\xi_j = -\sign(\hat c_{j,1})\infty.\]
If it happens that also $j\in \mathbb{J}^2_{s^*}\cup \mathbb{J}^3_{s^*}$, we have that for $m=2$ and $m=3$, $m\lambda_j-1>0$ and again the limit of both \eqref{E3.8} and \eqref{E3.9} is zero. So we have that
\[\lim_{x\to x_j}\partial_n \sum_{m=1}^3\hat c_{j,m} r_j^{m\lambda_j}\sin(m\lambda_j\theta_j)\xi_j = -\sign(\hat c_{j,1})\infty.\]
If $s^*>2$, also on this corner
\[\lim_{x\to x_j}\partial_n\bar\varphi_{reg}(x)=0,\]
and,  trivially
\begin{eqnarray}
\lim_{x\to x_j}\partial_n\bar\varphi(x)&=&
\lim_{x\to x_j}\left( \partial_n\bar\varphi_{reg}(x)+
\partial_n \sum_{m=1}^3\hat c_{j,m} r_j^{m\lambda_j}\sin(m\lambda_j\theta_j)\xi_j\right)\notag\\
&=&-\sign(\hat c_{j,1})\infty.
\label{E3.12}
\end{eqnarray}
If $s^*=2$, as we said at the beginning of the proof, $\partial_n\bar\varphi_{reg}\in H^{1/2}(\Gamma)$, so it needs not be even a bounded function. Nevertheless, since the singular part behaves like a negative power of $r_j$, this term dominates and we also have that \eqref{E3.12} holds.

As a consequence, there exists $\rho_j>0$ such that for $x\in\Gamma$ with $|x-x_j|<\rho_j$ either
$\proj_{[a,b]}\partial_n\bar\varphi\equiv a$ or $\proj_{[a,b]}\partial_n\bar\varphi\equiv b$
depending on the sign of $\hat c_{j,1}$. So the control is flat near non-convex corners.
This, together with the projection formula \eqref{E3.1} and \eqref{E3.11} implies that the optimal control belongs to $W^{1-1/p,p}(\Gamma)$. Finally, the regularity of the optimal state $\bar y$ follows from Lemma \ref{L2.3}.
\qquad\end{proof}

\begin{remark}\label{R3.5}
We would like to remark that the case of having $\hat c_{j,1}=0$ can be seen as a ``rare'' case in practice (although this can happen; see Example \ref{Ex3.8} below). So the ``normal'' case is that $\mathbb{H}^m_s=\emptyset$ for $m=2,3$. In this case, in the choice of $p$ made in \eqref{E3.6} the indexes $m=2,3$ are excluded, and hence $p$ will only depend $\max\{\omega_j:\ \omega_j<\pi\}$, so we get the same regularity for the control as that obtained in \cite{Casas-Raymond2006} for convex domains.
\end{remark}

To describe more accurately the regularity of the state and the control, we  have to introduce some further notation. Consider the coefficients $\hat c_{j,m}$ obtained in Theorem \ref{T3.2} and define the coefficients
\begin{equation}\label{E3.13}a_{j,m} =
\left\{\begin{array}{cl}
\displaystyle\frac{- m\lambda_j \hat c_{j,m}}{\nu}& \mbox{ if } 0\in[a,b]\\ \\
0 & \mbox{ if }0\not\in[a,b]
\end{array}
\right.
\end{equation}
and the functions
\begin{equation}\label{E3.14}s_{j,m}(\theta_j)=\frac{(-1)^{m+1}-\cos((m\lambda_j-1)\omega_j)}{\sin((m\lambda_j-1)\omega_j)}\sin\left( (m\lambda_j-1)\theta_j\right) + \cos \left( (m\lambda_j-1)\theta_j\right) .\end{equation}
The sets $\mathbb{H}^m_s$ will be used now to describe the singular part of the state and the control: if $j\in \mathbb{H}^m_s$, then
 $\xi_j r_j^{m\lambda_j-1}\not\in W^{1,s}(\Omega)$ and $\xi_jr_j^{m\lambda_j-1}\not\in W^{1-1/s,s}(\Gamma)$. Regarding \eqref{E3.16}, we also mention that if $\lambda_j \leq 1-2/s$, then $\xi_j r_j^{\lambda_j}\not\in W^{1,s}(\Omega)$.
\begin{theorem}\label{T3.6}
Assume further that  $ab\neq 0$.
Then there exist a unique $\bar u_{reg}\in W^{1-1/s^*,s^*}(\Gamma)$, a unique  $\bar y_{reg}\in W^{1,s}(\Omega)$, for all $s\leq s^*$, $s<p_D$, and unique real numbers $(c_j)_{\lambda_j \leq 1-2/s}$ such that
\begin{equation}\label{E3.15}\bar u(x)=\bar u_{reg}
+\sum_{m=1,3} \sum_{j\in \mathbb{H}^m_{s^*}}a_{j,m}\xi_j r^{m\lambda_j-1} +
\sum_{j\in \mathbb{H}^2_{s^*}}\chi_j a_{j,2}\xi_j r^{2\lambda_j-1}
\end{equation}

\begin{equation}\label{E3.16}\bar y = \bar y_{reg}
+\sum_{m=1}^3 \sum_{j\in \mathbb{H}^m_s}a_{j,m} \xi_j r_j^{m\lambda_j-1} s_{j,m}(\theta_j) +\sum_{\lambda_j \leq 1-2/s} c_j\xi_j r_j^{\lambda_j}\sin(\lambda_j \theta_j)\end{equation}
where the $\chi_j$ are the jump functions at the corners defined in \eqref{E2.10}, the $s_{j,m}(\theta)$ are defined in \eqref{E3.14} and the $\xi_j$ are the cut-off functions.
\end{theorem}

\begin{proof}
From the considerations in the proof of Theorem \ref{T3.4} we have that far from the corners with index $j\in \cup_{m=1}^3 \mathbb{H}^m_{s^*}$, the optimal control is the projection of a function that is either regular enough or tends to a signed $\infty$ at one point, so it is clear that $\bar u_{reg}\in W^{1-1/s^*,s^*}(\Gamma)$.

Next we will check what happens in the neighborhoods of the corners with index
$j\in \cup_{m=1}^3 \mathbb{H}^m_{s^*}$.

If $0\not\in[a,b]$ then the control would also be flat in neighborhoods of the  corners with index
$j\in \cup_{m=1}^3 \mathbb{H}^m_{s^*}$ again because the normal derivative of the adjoint state is continuous near the corner and $0$ at the corner. Then \eqref{E3.15} holds with $a_{j,m}=0$.

If $a<0<b$, then the optimal control will coincide with $\partial_n\bar\varphi/\nu$. Using formulas \eqref{E3.8} and \eqref{E3.9} and taking into account the definition of the jump functions on the corners $\chi_j$ \eqref{E2.10}, we have that there exists $\rho_j>0$ such that for all $x\in\Gamma$ such that $|x-x_j|<\rho_j$
\begin{equation}\label{E3.17}\bar u(x)=\frac{1}{\nu}\partial_n\bar\varphi(x) =
\left\{\begin{array}{cl}
\displaystyle{\frac{1}{\nu}\partial_n\bar\varphi_{reg}(x) - \frac{\hat c_{j,1}\lambda_j}{\nu}r_j^{\lambda_j-1}}
& \mbox{ if } j\in \mathbb{H}^1_{s^*},\\ \\
\begin{array}{rl}
\displaystyle{ \frac{1}{\nu}\partial_n\bar\varphi_{reg}(x)} &\displaystyle{ - \chi_j\frac{\hat c_{j,2}2\lambda_j}{\nu}r_j^{2\lambda_j-1}}\\&
\displaystyle{ - \frac{\hat c_{j,3}3\lambda_j}{\nu}r_j^{3\lambda_j-1}}
\end{array}
& \mbox{ if }  j\in \mathbb{H}^2_{s^*}\cup\mathbb{H}^3_{s^*},
\end{array}\right.
\end{equation}
and \eqref{E3.15} holds for $a_{j,m} = -m\hat c_{j,m}\lambda_j/\nu$.

Let us finally check \eqref{E3.16}.
We will write $\bar y = y_1+y_2$, where \[-\Delta y_1=0\mbox{ in }\Omega,\ y_1=\bar u_{reg}\mbox{ on }\Gamma\]
and
\[-\Delta y_2=0\mbox{ in }\Omega,\ y_2=\sum_{m=1,3} \sum_{j\in \mathbb{H}^m_{s^*}}a_{j,m}\xi_j r^{m\lambda_j-1}
+\sum_{j\in \mathbb{H}^2_{s^*}}\chi_j a_{j,2}\xi_j r^{2\lambda_j-1}
\mbox{ on }\Gamma.\]
Using Lemma \ref{L2.3} we have that
$y_1\in W^{1,s}(\Omega)$ for $s\leq s^*$, $s<p_D$.
From Lemma \ref{L2.5} we have that
 there exist a unique $y_{2,reg}\in W^{2,p}(\Omega)$, $p$ defined in \eqref{E3.6}, and unique real numbers $c_{j,m}$ such that
\begin{eqnarray*}
y_2 &= &y_{2,reg}+\sum_{m=1}^3\sum_{j\in \mathbb{H}^m_{s^*}}a_{j,m}\xi_j r_j^{m\lambda_j-1}s_{j,m}(\theta_j) \\
&&+\sum_{m=1}^3\sum_{j\in \mathbb{J}^m_{p}}c_{j,m}\xi_j r_j^{m\lambda_j}\sin(m\theta_j)
\end{eqnarray*}
Since $p\geq 2$ and $s\leq s^* <\infty$, in dimension 2 we have thanks to usual Sobolev's imbedding $W^{2,p}(\Omega)\hookrightarrow W^{1,s}(\Omega)$ that $y_{2,reg}\in W^{1,s}(\Omega)$.

As we already mentioned, among the terms of the second addend, those which correspond to $\mathbb{H}^m_s$ are not in $W^{1,s}(\Omega)$. Notice that if $s^*<p_D$, and hence $s<p_D$, then there could be some terms in $W^{1,s}(\Omega)$, which we gather in a function $y_{a,reg}\in W^{1,s}(\Omega)$.

For the last addend, we have that $\xi_j r_j^{m\lambda_j}\sin(m\theta_j)\not\in W^{1,s}(\Omega)$ iff $m\lambda_j\leq 1-2/s$. Since $s\geq 2$, this excludes the case $m>1$. We gather all the other terms in a function $y_{b,reg}\in W^{1,s}(\Omega)$.

So finally we have that the \eqref{E3.16}  holds for $\bar y_{reg}=y_1+y_{2,reg}+y_{a,reg}+y_{b,reg}\in W^{1,s}(\Omega)$ and $c_j=c_{j,1}$.
\qquad\end{proof}

%

Let us present now the example announced in Remark \ref{R3.5}.
For the example, we want to remark that our results are also applicable in curvilinear polygons without many changes (see \cite[Th. 5.2.7]{Grisvard85}). The only thing to take into account is that if the angle $\omega_j$ between two curved arcs is grater than $\pi$, then we must impose also $s^*<\omega_j/(\omega_j-\pi)$ (this is not the case in the following example).
\begin{example}\label{Ex3.8}Let $\omega_1=3\pi/2$ and consider the curvilinear polygon $\Omega=\{x\in\mathbb{R}^2:\ 0<r<1,\ 0<\theta<\omega_1\}$, where $(r,\theta)$ are the usual polar coordinates. We have that  $\omega_2=\omega_3=\pi/2$, and hence $\lambda_1=2/3$, $\lambda_2=\lambda_3=2$. Suppose $y_\Omega\in L^\infty(\Omega)$, so we may choose any $s^*<+\infty$. We have $\mathbb{J}^1_{s^*}=\mathbb{J}^2_{s^*}=\{1\}$ and $\mathbb{J}^3_{s^*}=\emptyset$. We also have $p<3$. The adjoint state can be written as
\[\bar\varphi = \bar\varphi_{reg} +\hat c_{1,1}\xi_1 r^{2/3}\sin\left(\frac23\theta\right) + \chi_1\hat c_{1,2}\xi_1 r^{4/3}\sin\left(\frac43\theta\right)\]
where $\bar\varphi_{reg}\in W^{2,s^*}(\Omega)$ for all $s^*<+\infty$. Take $-a=b=1$ for instance.
If $\hat c_{1,1}\neq 0$, then $\mathbb{H}^m_{s^*}=\emptyset$ for $m=1,2,3$ and hence $\bar u = \bar u_{reg}\in W^{1-1/s^*,s^*}(\Gamma)$.

Define now
\begin{equation}y_\Omega(x)=\left\{\begin{array}{rl}
1& \mbox{ if }\theta <\omega_1/2\\
-1&\mbox{ if }\theta >\omega_1/2
\end{array}\right.\label{E3.18}\end{equation}
such that the problem is skew-symmetric with respect to the line with $\theta=\omega_1/2$. 
The skew-symmetry of the data suggests that the solution is skew-symmetric, i.~e. that the symmetric contribution with $r^{2/3}\sin\left(\frac23\theta\right)$ vanishes, $\hat c_{1,1}=0$ (result which we have confirmed numerically),
 and hence
\[\bar u = \bar u_{reg} + \chi_1 a_{1,2}\xi_1 r^{1/3}\]
so $\bar u\in W^{1-1/p,p}(\Gamma)$ for all $p<3$.
\end{example}

In Theorem \ref{T3.6} we have excluded the cases $a=0$ or $b=0$. These cases can be treated with the same techniques. Nevertheless, many cases may appear depending on which of the bounds is zero and the sign of the coefficients of the singular part $\hat c_{j,m}$. As an example, we will show how to treat some of these cases.
 We will discuss first what we think is  the ``generic'' case, and then a  seemingly more ``rare'' case. Without loss of generality suppose $a=0$, $b>0$.

\paragraph{Case 1} Take $j\in \mathbb{H}^1_{s*}$ and suppose $\hat c_{j,1}\neq 0$. Then in the expression for the normal derivative of the adjoint state \eqref{E3.17}, the term $\xi_j r_j^{\lambda_j-1}$ dominates the term $\partial_n\varphi_{reg}$, since $\xi_j r^{\lambda_j-1}\not\in W^{1-1/s^*,s^*}(\Gamma)$ and $\partial_n\bar\varphi_{reg}\in W^{1-1/s^*,s^*}(\Gamma)$. Therefore, if $\hat c_{j,1}<0$, we would have that $\partial_n\bar\varphi(x) \in[0,b]$ in a neighborhood of $x_j$, and hence $\bar u(x)$ can be computed as in \eqref{E3.17}. On the other hand, if $\hat c_{j,1}>0$, then $\partial_n\bar\varphi(x) \leq 0$ in a neighborhood of $x_j$, so $\bar u(x)\equiv 0$ in that neighborhood. This would be the case of taking Example \ref{Ex3.8} with the following data: $\omega_1=3\pi/4$, $a=0$, $b=1$ and either $y_\Omega\equiv 1$ or $y_\Omega\equiv -1$, which would give $\sign(c_{j,1})=-\sign(y_\Omega)$.

\paragraph{Case 2}  If $j\in \mathbb{H}^2_{s^*}$ and $\hat c_{j,2} <0$, then $\partial_n\bar\varphi(x) \in[0,b]$ in a neighborhood of $x_j$ on the side $\Gamma_{j}$, but $\partial_n\bar\varphi(x) \leq 0$ in a neighborhood of $x_j$ on the side $\Gamma_{j-1}$, so on $\Gamma_{j}$, $\bar u(x)$ would have the same expression as in \eqref{E3.17}, but, on $\Gamma_{j-1}$, $\bar u(x)$ would be flat near the corner $x_j$. This would be the case of taking Example \ref{Ex3.8} for $a=0$, $b=1$ and $y_\Omega$ defined in \eqref{E3.18}.


\section{More regular data}\label{sec:4}\setcounter{equation}{0}
Taking advantage of the regularity of the optimal state, we can obtain several results for more regular data. We will write some results that we think will be useful for the numerical analysis of problem (P). To be specific, to obtain error estimates for problems with regular data, we will need that for $y_\Omega\in H^1(\Omega)$, $\bar u\in H^{3/2-\varepsilon}(\Gamma)$ (cf. Corollary \ref{C4.2}) and  $W^{3,p}(\Omega)$ regularity of the regular part of the adjoint state if $y_\Omega\in W^{1,p}(\Omega)$, $p\geq 2$ (see Corollary \ref{C4.3}). Some other results can be obtained used the techniques exposed in Section \ref{sec:3}.

Suppose now that $y_\Omega\in H^{t^*}(\Omega)$, with $0<t^*\leq 1$ such that
\[(1+t^*)/\lambda_j\not\in\mathbb{Z}\ \forall j\in \{1,\ldots,M\}.\]
For $t> -1$ and $m\in \mathbb{Z}$ define
\[\tilde{\mathbb{J}}^m_t =\left\{j\in \{1,\ldots,M\}\mbox{ such that }0<m\lambda_j<1+t\right\}\]
Notice again that due to our choice of $t^*$, we only will deal with the cases $m=1,2,3$.
\begin{corollary}There exist  a unique function $\bar\varphi_{reg}\in H^{2+t^*}(\Omega)$ and unique real numbers $\hat c_{j,m}$ such that
\[\bar\varphi = \bar\varphi_{reg} + \sum_{m=1}^3\sum_{j\in \tilde{\mathbb{J}}^m_{t^*} }\hat c_{j,m}  \xi_j r_j^{m\lambda_j}\sin(m\lambda_j\theta_j).\]
\label{C4.1}
\end{corollary}
\begin{proof}Since $t^*>0$, there exists $s^*>2$ satisfying \eqref{E3.4} such that $y_\Omega\in L^{s^*}(\Omega)$, and hence we can apply Theorem \ref{T3.4} and we have that also $\bar y\in W^{1,q}(\Omega)\hookrightarrow H^1(\Omega)\hookrightarrow H^{t^*}(\Omega)$, where $q>2$ is defined in Theorem~\ref{T3.4}. The result follows directly from the adjoint state equation \eqref{E3.3} thanks to the regularity results in \cite[\S23.C]{Dauge1988}; see also \cite[Th. 2.4.3 and \S2.7]{Grisvard1992}. Notice that due to the conditions imposed on $t^*$, logarithmic terms do not appear in the development of the singular part.
\qquad\end{proof}

To describe the regularity of the optimal control and state, we first introduce the sets $\tilde{\mathbb{H}}^m_t$ in an analogous way as we did for the sets $\mathbb{H}^m_s$, the indexes being taken now in the sets $\tilde{\mathbb{J}}^m_t$ defined above instead of the sets $\mathbb{J}^m_s$ and considering the coefficients $\hat c_{j,1}$ obtained in Corollary \ref{C4.1}.
We next define the exponents
$t>0$, related to the regularity of the control, and $\tilde t>0$, related to the regularity of the state such that
\begin{eqnarray*}
&t\leq t^*,\ t<1,\ t< m\lambda_j-1\mbox{ if }j\in \tilde{\mathbb{H}}^m_t,\ m=1,2,3\\
& \tilde t\leq t,\ \tilde t <\lambda_1.
\end{eqnarray*}
The meaning of these bounds is the following. The regularity of the optimal control will be limited by the regularity of the data, the impossibility of having a control globally in $H^{3/2}(\Gamma)$ due to the corners and the bound constraints, and the singular behavior of the control at the convex corners or the ``special'' nonconvex corners that may lay in $\tilde{\mathbb{H}}^m_t$ for $m=2,3$. The regularity of the optimal state will be limited by the regularity of the control and the singular behavior at the nonconvex corners of the solution.
\begin{corollary}\label{C4.2}Suppose
that
\begin{equation}\label{E4.1}\sharp\partial_\Gamma\{x\in\Gamma:\bar u(x)=a\mbox{ or }\bar u(x)=b\} < +\infty
\end{equation}
(the number of points on the boundary in the topology of $\Gamma$ of the active set is finite). Then
the optimal control $\bar u$ belongs to $H^{1/2+t}(\Gamma)$ and $\bar y\in H^{1+\tilde t}(\Omega)$.
\end{corollary}
\begin{proof}The proof follows the lines of that of Theorem \ref{T3.4}. Since $\bar\varphi_{reg}\in H^{2+t^*}(\Omega)$ and $\bar\varphi_{reg}=0$ on $\Gamma$ its normal derivative will be in $H^{1/2+t}(\Gamma)$ provided $t\leq t^*$ and $t<1$ (the condition $\bar\varphi_{reg}=0$ on $\Gamma$ is needed for $t^*\geq 1/2$ to prove that the normal derivative tends to zero at the corners, and hence it is continuous; notice that for $t^*=1$ this continuity is not enough to have that the normal derivative is in $H^{3/2}(\Gamma)$.)

This $H^{1/2+t}(\Gamma)$ regularity is not affected by the projection formula \eqref{E3.1} because $t<1$ and we are supposing \eqref{E4.1}. The same happens with the singular terms such that $\hat c_{j,1}\neq 0$ and $\lambda_j<1$. The rest of the singular terms in the expression for the normal derivative of the adjoint state will be in $H^{1/2+t}(\Gamma)$ since $t<m\lambda_j-1$ (see equations \eqref{E3.8} and \eqref{E3.9} for the expression of the normal derivatives of the singular part.)

Let us prove that $\bar y\in H^{1+\tilde t}(\Omega)$. Since $\bar u\in H^{1/2+t}(\Gamma)$, there exists some $U\in H^{1+t}(\Omega)$ such that $U = u$ on $\Gamma$. Moreover, $\Delta U \in H^{-1+t}(\Omega)$. So $z = \bar y-U$ is the solution of the boundary value problem
\[-\Delta z = \Delta U\mbox{ in }\Omega,\ z= 0\mbox{ on }\Gamma.\]
Using  the regularity results in \cite[\S23.C]{Dauge1988} (see also \cite[Th. 2.4.3 and \S2.7]{Grisvard1992}) we have that there exist a unique $z_{reg}\in H^{1+t}(\Omega)$ and  unique coefficients $c_{j,m}$ such that
\[z = z_{reg} + \sum_{m = 1}^3\sum_{j\in\tilde{\mathbb{J}}^m_{-1+t}} c_{j,m}\xi_j r^{m\lambda_j}\sin(m\lambda_j\theta_j).\]
Since $\tilde t <\lambda_1$, the singular part is in $H^{1+\tilde t}(\Omega)$, and so is the optimal state.
\qquad\end{proof}

Finally, we will describe the adjoint state for even more regular data. In the rest of this section we will suppose $y_\Omega\in W^{1,p^*}(\Omega)$, $p^*\geq 2$.

For $p > 1$ and $m\in \mathbb{Z}$ define
\[\mathbb{J}^m_{1,p} =\left\{j\in \{1,\ldots,M\}\mbox{ such that }0<m\lambda_j<3-\frac{2}{p}\mbox{ and } m\lambda_j\not\in\mathbb{Z}\right\}.\]
Now we have that   $\mathbb{J}^m_{1,p}=\emptyset$ if $m>5$. We have to add the condition $m\lambda_j\not\in\mathbb{Z}$ otherwise logarithmic terms may appear. Define also:
\[\mathbb{L}^m_{1,p} =\left\{j\in \{1,\ldots,M\}\mbox{ such that }0<m\lambda_j<3-\frac{2}{p}\mbox{ and } m\lambda_j\in\mathbb{Z}\right\}.\]
A direct calculation gives us that $\mathbb{L}^m_{1,p}=\emptyset$ if $m=2$ or $m\geq 4$, $\mathbb{L}^1_{1,p}\subset\{j:\ \omega_j = \pi/2\}$ and $\mathbb{L}^3_{1,p}\subset\{j:\ \omega_j = 3\pi/2\}$, and hence $m\lambda_j=2$ if $j\in \mathbb{L}^m_{1,p}$.

Consider now $p\geq 2$ such that, for $m=1,2,3$
\begin{equation}\label{E4.2}
p\leq p^*,\ p<p_D,\ p<\frac{2}{2-m\lambda_j}\mbox{ if }j\in \mathbb{H}^m_{s^*}\ \forall {s^*}<\infty.
\end{equation}
In addition, we need to to assume
\[\frac{3p-2}{\lambda_j p}\not\in \mathbb{Z}\ \forall j\in \{1,\ldots,M\}.\]
With this notation, we have the following result.
\begin{corollary}
There exist  a unique function $\bar\varphi_{reg}\in W^{3,p}(\Omega)$ and unique real numbers $\hat c_{j,m}$ and $\hat d_{j,m}$ such that
\begin{eqnarray*}
\bar\varphi & = &\bar\varphi_{reg} + \sum_{m=1}^5\sum_{j\in \mathbb{J}^m_{1,p} }\hat c_{j,m} \xi_j r_j^{m\lambda_j}\sin(m\lambda_j\theta_j) \\
&&+ \sum_{m=1,3}\sum_{j\in \mathbb{L}^m_{1,p} }\hat d_{j,m}\xi_j  r_j^{2} \left(\log(r_j)\sin(2\theta_j) + \theta_j\cos(2\theta_j) \right).
\end{eqnarray*}
\label{C4.3}
\end{corollary}
\begin{proof}
Since $p^*\geq 2$, $y_\Omega\in L^{s^*}(\Omega)$ for any $s^*<\infty$, and hence we can apply Theorem \ref{T3.4} and we have that also $\bar y\in W^{1,p}(\Omega)$. The result follows directly from the adjoint state equation \eqref{E3.3} thanks to the regularity result \cite[Th. 5.1.3.5]{Grisvard85}
\begin{eqnarray*}
\bar\varphi & = &\bar\varphi_{reg} + \sum_{m=1}^5\sum_{j\in \mathbb{J}^m_{1,p} }\hat \xi_j r_j^{m\lambda_j}\sin(m\lambda_j\theta_j) \\
&&+ \sum_{m=1,3}\sum_{j\in \mathbb{L}^m_{1,p} }\hat d_{j,m}\xi_j  r_j^{m\lambda_j} \left(\log(r_j)\sin(m\lambda_j\theta_j) + \theta_j\cos(m\lambda_j\theta_j) \right).
\end{eqnarray*}
and using that $m\lambda_j=2$ if $j\in \mathbb{L}^m_{1,p}$.
\qquad\end{proof}

\section{\label{sec:5}Problems without control constraints}
For problems without control constraints, we obtain similar results. Indeed, for convex domains and data $y_\Omega\in L^{s^*}(\Omega)$, $s^*>2$, it is obvious from Theorem \ref{T5.1} below that the optimal control is a bounded function and hence all the results stated before apply. Nevertheless, for nonconvex domains, we will not obtain a continuous control and the singularities must be taken into account near all the corners. Therefore, the indexes for the expansion of the singular parts must be taken running through all the sets $\mathbb{J}^m_s$, and not only through the sets $\mathbb{H}^m_s$.
\begin{theorem}\label{T5.1}Suppose now that $-a=b=\infty$ and $y_\Omega\in L^{s^*}(\Omega)$, $s^*\geq 2$. Then there exists a unique $\bar\varphi_{reg}\in W^{2,s}(\Omega)$, $\bar u_{reg}\in W^{1-1/s,s}(\Gamma)$, $\bar y_{reg}\in W^{1,s}(\Omega)$, for all $s\leq s^*$, $s<p_D$, and unique real numbers $(\hat c_{j,m})_{j\in \mathbb{J}^m_{s}}$  and $(c_j)_{\lambda_j<1-2/s}$ such that
\begin{equation}\label{E5.1}\bar\varphi = \bar\varphi_{reg} +\sum_{m=1}^3\sum_{j\in \mathbb{J}^m_{s}} \hat c_{j,m} r_j^{m\lambda_j}\sin(m\lambda_j\theta_j)\xi_j\end{equation}
\begin{equation}\label{E5.2}\bar u(x)=\bar u_{reg}+
\sum_{m=1,3}\sum_{j\in \mathbb{J}^m_s} a_{j,m} \xi_j r^{m\lambda_j-1} +
\sum_{j\in \mathbb{J}^2_s} \chi_j a_{j,m}\xi_j r^{2\lambda_j-1} \end{equation}
\begin{equation}\label{E5.3}\bar y = \bar y_{reg}+
\sum_{m=1}^3\sum_{j\in \mathbb{J}^m_{s}} a_{j,m}\xi_j r_j^{m\lambda_j-1} s_{j,m}(\theta_j) +\sum_{\lambda_j<1-2/s}c_j\xi_j r_j^{\lambda_j}\sin(\lambda_j\theta_j)\end{equation}
where $a_{j,m}$ and $s_{j,m}(\theta)$ are given by the formulas \eqref{E3.13} and \eqref{E3.14}.
\end{theorem}
\begin{proof}The proof is very similar to those of theorems \ref{T3.2}, \ref{T3.4} and \ref{T3.6}. We will only emphasize on the main difference: at the beginning of the proof of Theorem \ref{T3.2} we used that the optimal control was bounded to obtain that the optimal state was a function in $L^{s^*}(\Omega)$. Now the optimal control is not bounded, so we use a bootstrapping argument to show that $\bar y\in L^s(\Omega)$ for $s\leq s^*$, $s<p_D$. The result follows then using the same techniques as before. Notice that now we do not need the sets $\mathbb{H}^m_s$, since we do not have to exclude in the expression of the singular part of the control the corners where the normal derivative of the adjoint state is not bounded.

In a first step we have that $\bar u\in L^2(\Gamma)$, and hence $\bar y\in H^{1/2}(\Omega)\subset L^4(\Omega)$.
If $s^*\leq 4$ the proof is complete since $4 < p_D$ for any polygonal domain.

Suppose $s^*>4$.
The normal derivative of the regular part of the adjoint state is in $W^{1-1/4,4}(\Gamma)$, but now, since we have no control constraints, we have to take into account the normal derivative of the singular part near the non-convex corners. We have so far that the optimal control can be written as the sum of a regular part, which is in $W^{1-1/4,4}(\Gamma)$ plus a singular part, that behaves as $r_1^{\lambda_1-1}$. For the regular part we  apply Lemma \ref{L2.3} and for the singular part we apply Lemma \ref{L2.5}, and we have that the optimal state can be written as the sum of a regular part which is in $W^{1,4}(\Omega)\subset L^{s^*}(\Omega)$ plus a  singular part that behaves at worst as $r_1^{\lambda_1-1}\xi_1\in L^s(\Omega)$ for all $s<2/(1-\min\{1,\lambda_1\})=p_D$. So we have that $\bar y\in L^s(\Omega)$ for all $s\leq s^*$, $s<p_D$.
\qquad\end{proof}

We will 
 finish this section stating some regularity results of the optimal solution in the unconstrained case in some special situations.

\begin{corollary}\label{C5.2}
Suppose the assumptions of Theorem \ref{T5.1} are satisfied. Then, for all $p\leq s^*$, $p<p_\Omega$ we have that $\bar\varphi\in W^{2,p}(\Omega)$, $\bar u\in W^{1-1/p,p}(\Gamma)$ and $\bar y\in W^{1,p}(\Omega)$.
\end{corollary}
\begin{proof}Since $p\leq s^*$ and $p<p_\Omega<p_D$, the regular parts of the involved functions satisfy $\bar u_{reg}\in W^{1-1/p,p}(\Gamma)$, $\bar y_{reg}\in W^{1,p}(\Omega)$ and $\bar\varphi_{reg}\in W^{2,s}(\Omega)$ due to Theorem \ref{T5.1}.


On the other hand, the assumption $p<p_\Omega$ implies $\xi_1r_1^{\lambda_1-1}\in W^{1,p}(\Omega)$, and hence obviously $\xi_1r_1^{\lambda_1}\in W^{2,p}(\Omega)$ and $\xi_1r_1^{\lambda_1-1}\in W^{1-1/p,p}(\Gamma)$. Since these are the worst terms we may find in the singular parts, the proof is complete.
\qquad\end{proof}

With the same techniques of Section \ref{sec:4}, (Corollary \ref{C4.1} and \ref{C4.2}) we can obtain the following result. 
\begin{corollary}Suppose the assumptions of Theorem \ref{T5.1} are satisfied and $y_\Omega\in H^{t^*}(\Omega)$ for some $t^*\leq 1$. Define $t>0$ such that
\begin{equation*}
t\leq t^*, t<1,\ t<\lambda_1-1.
\end{equation*}
Then $\bar u\in H^{1/2+ t}(\Gamma)$ and $\bar y\in H^{1+ t}(\Omega)$. 
\end{corollary}

\bibliography{references_for_constrained_Dirichlet}
\bibliographystyle{siam}
\end{document}